%% file: main.tex
\documentclass[letterpaper, 10 pt, conference]{ieeeconf}

\IEEEoverridecommandlockouts
\usepackage{balance}
\usepackage{color}
\usepackage{enumerate}
\usepackage{cite}
\usepackage{empheq}
\usepackage{mathrsfs}
\usepackage{amsmath}
\usepackage{cases}
\usepackage{subfig}
\usepackage{mathtools}
\usepackage[font=small,labelfont=bf]{caption}
\usepackage{float}
\usepackage[bookmarks=true]{hyperref}
\usepackage{algorithm,algpseudocode}

\usepackage{hyperref,url}

\include{macros-abbrev}

\makeatletter
\let\proof\@undefined                        
\let\endproof\@undefined                  
\makeatother
\usepackage{graphicx,amssymb,amstext,amsmath,amsthm}

\newtheorem{prop}{Proposition} 

\newtheorem{thm}{Theorem}
\newtheorem{assumption}{Assumption}
\newtheorem{lemma}{Lemma}
\newtheorem{defn}{Definition}
\newtheorem{rem}{Remark}
\newtheorem{problem}{Problem}

\let\oldReturn\Return
\renewcommand{\Return}{\State\oldReturn}

\let\oldbibliography\thebibliography
\renewcommand{\thebibliography}[1]{%
  \oldbibliography{#1}%
}

\newcommand{\intreal}{\mathbb{I}\real}
\renewcommand{\P}{\mathcal{P}}
\newcommand{\Q}{\mathcal{Q}}
\newcommand{\X}{\mathcal{X}}
\newcommand{\U}{\mathcal{U}}
\renewcommand{\H}{\mathcal{H}}

\renewcommand{\S}{\mathcal{S}}
\newcommand{\N}{\mathcal{N}}
\newcommand{\queue}{Q}

\DeclareMathOperator{\conv}{conv}
\DeclareMathOperator{\pop}{pop\_front}

\DeclareMathOperator{\pushfront}{push\_front}
\DeclareMathOperator{\pushback}{push\_back}
\DeclareMathOperator{\bisect}{bisect}
\DeclareMathOperator{\checked}{checked}
\let\mid\undefined
\DeclareMathOperator{\mid}{mid}

\newcommand\oprocendsymbol{\hbox{$\bullet$}}
\newcommand\oprocend{\relax\ifmmode\else\unskip\hfill\fi\oprocendsymbol}

\begin{document}

\title{\Large Computing Controlled Invariant Sets of Nonlinear
  Control-Affine Systems}
\author{Scott Brown \and
  Mohammad Khajenejad \and
  Sze Zheng Yong \and
  Sonia Mart{\'\i}nez
  \thanks{S. Brown, M, Khajenejad and S. Mart{\'\i}nez are with the Mechanical and Aerospace Engineering Department,
    University of California San Diego, San Diego, CA, USA.
    \texttt{\{sab007,mkhajenejad,soniamd\}@ucsd.edu}}\and
  \thanks{S. Z. Yong is with the Mechanical and Industrial Engineering
    Department, Northeastern University.
    \texttt{s.yong@northeastern.edu}}}

\date{}

\maketitle

\begin{abstract}
  In this paper, we consider the computation of controlled invariant
  sets (CIS) of discrete-time nonlinear control-affine systems. We
  propose an iterative refinement procedure based on polytopic
  inclusion functions, which is able to approximate the maximal
  controlled invariant set to within a guaranteed precision. In
  particular, this procedure allows us to guarantee the invariance of
  the resulting near-maximal CIS while also computing sets of control
  inputs which enforce the invariance. Further, we propose an
  accelerated version of this procedure which refines the CIS by
  computing backward reachable sets of individual components of set
  unions, rather than all at once. This reduces the total number of
  iterations required for convergence, especially when compared with
  existing methods. Finally, we compare our methods to a sampling
  based approach and demonstrate the improved accuracy and faster
  convergence.
\end{abstract}

\section{Introduction}
Invariance is an important concept for ensuring robustness and safety
of control systems. For a dynamical system, a set is (\emph{forward})
\emph{invariant} if every trajectory starting in that set remains in
the set for all time. For control systems, this notion can be
generalized with the determination of a control input which is able to
render a set invariant, leading to the notion of a \emph{controlled
  invariant set} (CIS). For systems which are subject to uncertainty
or noise, the concept of a \emph{robust controlled invariant set}
(RCIS) is critical for safety, as it guarantees invariance in the
presence of disturbances.

Controlled invariant sets have been thoroughly studied for linear
systems, e.g., \cite{blanchini1999invariance,fiacchini2018invariant,
  rungger2017invariant,anevlavis2019, anevlavis2021}. Many of these
methods employ iterative procedures based on a one-step backward
operator \cite{fiacchini2018invariant, rungger2017invariant} to find
backward reachable sets of the system for computing the CIS with high
precision.
In order to improve the computation time for high-dimensional
systems, other non-iterative techniques have also been proposed
which rely on lifting to a higher dimensional space to compute the CIS
in closed form and projecting the resulting set to the original
domain, e.g., \cite{anevlavis2019, anevlavis2021}.

On the other hand, determining controlled invariant sets of nonlinear
systems remains a significant challenge. Some works employ convex or
zonotopic approximations, e.g.,
\cite{fiacchini2010invariant,schafer2022scalable}, in order to reduce
the computational complexity. However, as the maximal controlled
invariant sets are nonconvex in general, these methods can be overly
conservative.

Another related work pertains to the study of invariant sets of
\emph{switched} systems, e.g.,
\cite{jang2022invariant,li2018invariant,bravo2005invariant}, where the
input controls switching between a finite number of modes. In that
case, the input is easily determined by considering all possible modes
and selecting those which lead to invariance
\cite{li2018invariant}. By sampling a continuous set of inputs, these
methods can be applied to more general nonlinear systems, but the
accuracy and scalability may be limited \cite{bravo2005invariant}.
These methods result in sets that are guaranteed to be invariant, but
the sampling is an additional source of computational complexity as it
must be fine enough to properly capture the nonlinear behavior of the
system. As a result, these methods are difficult to apply to systems
with multiple inputs.

In this paper, we propose two iterative algorithms to compute the
near-maximal (non-convex) controlled invariant set of control affine
systems up to a guaranteed precision, \emph{without} sampling the set
of allowable inputs. Inspired by \cite{li2018invariant}, our methods
use a bisection approach to over- and under-approximate the one-step
backward reachable set operator. The main idea of the approach is to
compute the \emph{forward} reachable set of the region of interest,
which is described by a union of intervals. The parts of this forward
set that are entirely contained in the original set are used as an
under-approximation of the backward reachable set. To check the
containment, we use polyhedral over-approximations of reachable sets,
which allows us to cast the problem of determining the control input
as a translation of a polyhedron. This technique allows us to
determine a \emph{continuous}, rather than sampled, set of invariance
enforcing control inputs. In addition, we leverage the structure of
set unions in the accelerated version, which has the potential to
significantly reduce the total number of iterations required for
convergence, when compared with existing methods.

\subsection{Notation}
Let $\integernonnegative$ be the set of nonnegative integers, and
$\real^n$ and $\real^{n \times p}$ be the $n$-dimensional Euclidean
space, and the set of $n \times p$ matrices, respectively. By means of
$\mathcal{B}_r$, we denote the $\infty$-norm hyperball of radius
$r$. We make use of $\intreal^n$ to denote the collection of
(multidimensional) intervals of $\real^n$, and denote its elements as
$[x]\triangleq [\underline{x},\overline{x}] \in \intreal^n$ with lower
and upper bounds $\underline{x} $ and $\overline{x}$. The function
$w([x])$ measures interval width (i.e.,
$w([x])=\|\overline{x}-\underline{x}\|$ with any vector norm
$\|\cdot\|$), $\mid([x])$ selects the midpoint of $[x]$ (i.e.,
$\mid([x])=\frac{1}{2}(\overline{x}+\underline{x})$), and a function
in brackets $[f]([x])$ represents an interval inclusion function
(cf. Definition \ref{def:inclusion}). The operator $\oplus$ denotes
the Minkowski sum of two sets, and $\ominus$ denotes the Pontryagin
difference. For a function $f: \real^n \to \real^m$ and a set
$\mathcal{S} \subset \real^n$, $f_0(\mathcal{S})$ denotes the image of
$\mathcal{S}$ under $f_0$.  For a bounded polyhedron (polytope) $\P$,
we use $H_\P$ and $b_\P$ to denote the components defining the
halfspace representation $\P = \{x \ : \ H_\P x \le b_\P\}$. The
symbol $V_\P$ denotes the vertices of $\P$. Finally, $\conv\X$ denotes
the convex hull of the set $\X$.

\section{Preliminaries}
This section introduces preliminary notions that will be used throughout
the paper. We begin by defining inclusion functions, which are
critical for tractably approximating the images of sets under
nonlinear functions.
\begin{defn} \label{def:inclusion}
  Given a function $f: \real^n \to \real^m$, an \emph{inclusion
    function} is an interval function $[f]: \intreal^n \to \intreal^m$
  that satisfies
  \begin{align*} [f]([x]) \supseteq f([x]), \ \forall [x] \in
    \intreal^n,
  \end{align*}
  where $f([x])$ denotes the exact image of $[x]$ under $f$.
\end{defn}
The reader is referred to \cite[Section 2.4]{jaulinapplied} and \cite{khajenejad2023tight}
for a thorough discussion of different types of inclusion functions,
such as natural, centered, etc., as well as mixed-monotone
decomposition-based inclusion functions. The results of this paper are
not specific to any one type of inclusion function. However, different
inclusion functions may produce more or less precise approximations
depending on the system.

\subsection{Translating Polytopes}
This section defines two different operations on polytopes:
translating one into another, and translating one so it intersects
another. These will be used in our algorithm to determine safe control
inputs.

\begin{defn}
    Given two polytopes $\P$ and $\Q$, the set of translations of $\P$
    that \emph{insert} $\P$ into $\Q$ is denoted by
  \begin{align*}
  \mathcal{I}(\P, \Q) \triangleq \{r \in \real^n : \P \oplus \{r\} \subseteq \Q\}.
  \end{align*}
Similarly, the set of translations of $\P$ that
  \emph{overlap} $\P$ with $\Q$ is defined as
  \begin{align*}
\mathcal{O}(\P, \Q)
  \triangleq \{s \in \real^n : \P \oplus \{s\} \cap \Q \ne \emptyset \}.
  \end{align*}
\end{defn}

\begin{prop}[\hspace{-0.001em}{\cite[Theorem 2.3]{gilbert1998}}]\label{prop:into}
  For polytopes $\P$ and $\Q$,
  \begin{align*}
    \mathcal{I}(\P, \Q) = \{r \in \real^n : H_\Q r \leq b_\Q - \beta \},
\end{align*}
where $\beta_i = \max_{v\in V_\P} (H_\Q)_i v$.  If $\P$ cannot be
embedded into $\Q$, then this results in
$\mathcal{I}(\P,\Q) = \emptyset$.
\end{prop}
\begin{prop}\label{prop:touch}
  Given two polytopes $\P$ with $\Q$,
\begin{align*}
  \mathcal{O}(\P, \Q) = \left\{s \in \real^n : \begin{bmatrix}H_\Q \\ - H_\P \end{bmatrix} s \leq \begin{bmatrix}b_\Q-\alpha \\ b_\P-\gamma \end{bmatrix} \right\},
\end{align*}
where
$\alpha_i = \min_{v\in V_\P} (H_\Q)_i v$ and $\gamma_i = \min_{v\in
  V_\Q} (H_\P)_i v$.
\end{prop}
\begin{proof}
  See Appendix~\ref{sec:prop-proof}.
\end{proof}
\section{Invariance Control Problem}
Here we introduce the class of systems under consideration and define
the concept of controlled invariance. Specifically, we consider nonlinear control
affine systems of the form
\begin{align}
  \label{eq:sys}
  x_{k+1} = f(x_k, u_k) \triangleq f_0(x_k) + \sum_{i = 1}^{m}g_i(x_k)u_{i,k},
\end{align}
where $x \in \real^n$ is the state and $u \in \U \subset \real^m$ is
the input.  We assume that $\U$ is a compact interval, and we will
restrict our attention to a region of interest,
$\Omega \subset \real^n$, which we assume to be given as a finite
union of compact intervals.
\begin{assumption}[]\label{ass:lipschitz}
  The functions $f_0$, and $g_1$, $\dots$, $g_m$ are Lipschitz continuous, i.e.,
  \begin{align*}
    \forall i \in \{1,\dots,m\}, \ \exists &L_i \text{ s.t. } \|g_i(x) - g_i(y)\| \le L_i \|x - y\|,\\
    \text{and }&L \text{ s.t. }  \|f_0(x) - f_0(y)\| \le L \|x - y\|,
  \end{align*}
  for every $x, y \in \Omega$.
\end{assumption}
Next we define the precise notions of invariance which we will
consider in this work.
\begin{defn}
  A set $\X \subset \real^n$ is \emph{controlled invariant} with
  respect to the dynamics \eqref{eq:sys} if for every
  $x_0 \in \X$, there exists an input $u$ such that
  $f(x_0, u) \in \X$.
\end{defn}

There are many computational challenges associated with computing
controlled invariant sets, many of which arise due to the infinite
precision required to adequately handle regions near the boundary of
the set. As such, it is convenient to modify the definition to
incorporate a robustness margin $r$.
\begin{defn}
  A set $\X \subset \real^n$ is \emph{$r$-robustly controlled
    invariant} for system \eqref{eq:sys} if for every $x_0 \in \X$,
  there exists an input $u$ such that
  $f(x_0, u) \in \X \ominus \mathcal{B}_r$.
\end{defn}
Intuitively, in order to be robustly invariant, every point must be
mapped into the interior of the set, at some distance $r$ from the
boundary.

We are ready to formally state the problem which we aim to address in this paper.
\begin{problem}[Controlled Invariant Set Computation]
  For a system \eqref{eq:sys} that satisfies Assumption
  \ref{ass:lipschitz} and a region of interest
  $\Omega \subset \real^n$, given by a union of compact intervals,
  compute 
  the maximal controlled invariant set contained in $\Omega$ up to a
  guaranteed precision. Additionally, compute the corresponding set of
  control inputs that enforces this invariance.
\end{problem}

\begin{rem}
  {\rm For ease of exposition, we do not consider any noise or
    uncertainty in the system dynamics \eqref{eq:sys}. However, it is
    straightforward to include robustness to bounded noise in our
    algorithm by increasing the robustness margin $r$.} 
\end{rem}

\section{Computation of Controlled Invariant Sets}
This section starts by reviewing a well-known iterative procedure for
computing maximal controlled invariant sets
\cite{bertsekas1972invariant}. After this, we describe our main
contribution, which approximates the operator used in each iteration.

\begin{defn}
  The \emph{pre-set}, or the \emph{one-step backward reachable set} of
  a set $\Omega \subset \real^n$ is defined as
  \begin{align*}
    Q(\Omega) \triangleq \left\{x \in \real^n : \exists u \in \U \text{ s.t. } f(x, u) \in \Omega \right\}.
  \end{align*}
\end{defn}
We further define the operator
\begin{align*}
  I(\Omega) \triangleq Q(\Omega) \cap \Omega,
\end{align*}
which will enable computation of the maximal controlled invariant
set. Repeated application is denoted as
$I^i(\Omega) = I(I^{i-1}(\Omega))$, $i \in \integernonnegative$, with
$I^0(\Omega) = \Omega$. We also define
$I_r(\Omega) \triangleq Q(\Omega \ominus \mathcal{B}_r) \cap \Omega$.

\begin{lemma}[\hspace{-0.001em}{\cite[Special Case of Proposition 4]{bertsekas1972invariant}}]
  If $\Omega \subset \real^n$ is closed, then
  \begin{align*}
    I^\infty \triangleq \lim_{i\to\infty} I^i(\Omega)
  \end{align*}
  is the maximal controlled invariant set contained in $\Omega$.
\end{lemma}
This result is the basis for many iterative algorithms, e.g.,
\cite{fiacchini2018invariant, rungger2017invariant,
  bertsekas1972invariant}, but there are several computational
challenges when dealing with nonlinear systems. In addition to
operations involving \emph{backward} reachability, our algorithm also
utilizes operations related to \emph{forward} reachability, which we
define here.

\begin{defn}
  The \emph{one-step forward reachable set} of a set
  $\Omega \subset \real^n$ is defined as
  \begin{align*}
    P(\Omega) \triangleq \{x \in \real^n : \exists x_0 \in \Omega, u \in \U
    \text{ s.t. } x = f(x_0, u)\}.
  \end{align*}
\end{defn}
The following definition restricts the previous reachable set to that of a
particular input, which is used later for determining a
suitable controller.
\begin{defn}
  For a given $u\in\U$, the \emph{fixed-control one-step forward
    reachable set} of a set $\Omega \subset \real^n$ is defined as
  \begin{align*}
    P_u(\Omega) \triangleq \{x \in \real^n :  \exists x_0 \in \Omega \text{ s.t. } x = f(x_0, u) \}.
  \end{align*}
\end{defn}
From the definitions, it is clear that
$\bigcup_{u\in\U}P_u(\Omega) = P(\Omega)$.

\subsection{Polyhedral Approximation of Reachable Sets}
Our algorithm will employ polyhedral over-approximations of
$P(\Omega)$ and $P_u(\Omega)$ to determine feasible control inputs
that can lead to invariance.

To this end, we use a decomposition of the function $f$, which we will
show to satisfy certain properties. This decomposition will vary
depending on the interval $[x]$ under consideration. We first compute
$A$ and $\phi$ such that
\begin{align}\label{eq:f}
  f_0(x) =  Ax + \phi(x), \ \forall x \in [x],
\end{align}
decomposing $f$ into a linear term plus a remainder. This is always
possible (since we can let $A = 0$) and can be done in multiple
different ways. For example, if $f$ is differentiable, this can be
done via linearization about the midpoint of the interval. Another
possibility is that $f$ has a bounded Jacobian matrix on $[x]$, in
which case we can apply a Jacobian sign-stable decomposition
\cite[Proposition 2]{moh2022intervalACC} to compute \eqref{eq:f}. The
method of decomposition will affect the accuracy of the final
approximation, as we will discuss at the end of this section. Then, by
using an inclusion function $[\Phi] : \intreal^n \to \intreal^n$
satisfying $[\Phi]([x]) \supseteq \phi([x])$, we can guarantee that
\begin{align}\label{eq:f_over}
  f_0([x]) \subseteq  A[x] \oplus [\Phi]([x]),
\end{align}
where $A[x]$ denotes the exact (polytope) image of $[x]$ under the
linear map $A$.

We also decompose the individual input functions $g_i$. For an
inclusion function $[g_i]$, let $s_i = \mid([g_i]([x]))$ and
$[\Psi_i]([x]) = [g_i]([x]) \ominus \{s_i\}$, so that
\begin{align}\label{eq:g_i}
  g_i([x]) \subseteq \{s_i\} \oplus [\Psi_i]([x]),
\end{align}
and $[\Psi_i]([x])$ is centered at the origin. A centered
$[\Psi_i]([x])$ will result in a better approximation
later on. We will use the notation $S\in\real^{n\times m}$ and
$[\Psi]: \intreal^n \to \intreal^{n\times m}$ to denote the matrices
with columns $s_i$ and $[\Psi_i]$, respectively.

In order to guarantee the accuracy of the algorithm, we must first
compute bounds on the error of these
overapproximations. Assumption~\ref{ass:lipschitz} allows us to upper
bound the width of the resulting inclusion functions. \begin{prop}\label{prop:widths}
  Under Assumption~\ref{ass:lipschitz}, inclusion functions $[\Phi]$
  and $[\Psi_i]$, $i \in \until{m}$, and constants
  $\tilde{L}_i > 0, \ i \in \{0,\dots,m\}$ can be found so that
  \begin{align*}
    w([\Phi]([x]) \le \tilde{L}_0 w([x]) \ \text{ and } \
    w([\Psi_i]([x]) \le \tilde{L}_i w([x]).
  \end{align*}
  for every interval $[x] \subseteq \Omega$.
\end{prop}
\begin{proof}
  We can, for example, use inclusion functions based on mixed-monotone
  decompositions (cf. \cite[Lemma 1]{khajenejad2020simultaneousCDC})
  or the mean-value form (cf. \cite{jaulinapplied} and
  \cite[Eq. (3)]{li2018invariant}) that can be shown to satisfy these
  bounds.
\end{proof}
With these decompositions in mind, we can define our approximations of
the forward reachable sets. Let
\begin{align*}
  \overline{P}([x]) \triangleq A[x] \oplus [\Phi]([x])
  \oplus S\mathcal{U} \oplus [\Psi]([x])\mathcal{U}
\end{align*}
be the polyhedral overapproximation of the one-step forward reachable
set, and let
\begin{align*}
  \overline{P}_u([x]) \triangleq A[x] \oplus [\Phi]([x])
  \oplus Su \oplus [\Psi]([x])\mathcal{U}
\end{align*}
be the polyhedral overapproximation of the fixed-control reachable
set. The expression $S\U$ denotes the polytopic image of $\U$ under
the linear transformation $S$, and $[\Psi]([x])\U$ denotes the
interval-valued product of two interval matrices that can be computed
using interval arithmetic.

\begin{lemma}\label{lem:widths}
  Let
  $\rho \triangleq \tilde{L}_0 + \max_{1\le i\le m}
  \tilde{L}_iw(\U_i)$. Then for all intervals $[x]\subset \Omega$, the
  polytopic approximations of the forward reachable sets satisfy
  \begin{align*}
    P_u([x]) &\subseteq \overline{P}_u([x]) \subseteq P_u([x]) \oplus \mathcal{B}_{\rho w([x])}
  \end{align*}
  and
  \begin{align*}
    P([x]) &\subseteq \overline{P}([x]) \subseteq P([x]) \oplus \mathcal{B}_{\rho w([x])}.
  \end{align*}
\end{lemma}
\begin{proof}
  The inclusion $P_u([x]) \subseteq \overline{P}_u([x])$ is guaranteed
  by construction, since
  \begin{align*}
    P_u([x]) = A[x] \oplus \phi([x]) \oplus Su \oplus \sum_{i=1}^m
    \hat{g}_i([x]))u_i,
    \end{align*}
    $\phi([x]) \subseteq [\Phi]([x])$, and
    $\hat{g}_i([x]) \subseteq [\Psi_i]([x])$. The inclusion
    $\overline{P}_u([x]) \subseteq P_u([x]) \oplus \mathcal{B}_{\rho
      w([x])}$ follows from Proposition~\ref{prop:widths} and the
    definition of the Minkowski sum. The second statement is proved in
    the same way. 
  \end{proof}

\subsection{Interval Approximation of Pre-Sets}

We describe our main contribution next, which is a novel algorithm for
approximating the operator $I$ for systems of the form
\eqref{eq:sys}. Given a union of compact intervals $\Omega$, we
propose an iterative refinement procedure that approximates
$I(\Omega)$ by another union of compact intervals. Algorithm
\ref{alg:I} summarizes the main steps of this process, described in
detail next.

\begin{algorithm}
  \caption{$\underline{I}(\Omega)$}\label{alg:I}
  \begin{algorithmic}[1]
    \Require $\Omega$, $\varepsilon$
    \State $\queue \gets  \{\Omega\}$, $N \gets \emptyset$, $\underline{I} \gets \emptyset$, $\mathcal{E} \gets \emptyset$, $\U_I \gets \emptyset$
    \While{$\queue \neq \emptyset$}
      \State $[x] \gets \pop(\queue)$
      \State Compute $A$, $\Phi$, $s_i$, and $\Psi_i$ on $[x]$
      \If{$\overline{P}([x]) \cap \Omega = \emptyset$}
        \State $N \gets N \cup [x]$
      \ElsIf{$\exists u \in \U \text{ s.t. } \overline{P}_u([x]) \subseteq \Omega$} \label{line:u}
        \State $\underline{I} \gets \underline{I} \cup [x]$
	\State $\U_I \gets \U_I \cup ([x],S^\dagger(S\U([x])))$
      \ElsIf{$w([x]) \le \varepsilon$}
        \State $\mathcal{E} \gets \mathcal{E} \cup [x]$
      \Else
        \State $(l, r) \gets \bisect([x])$
        \State $\pushfront(\queue, l)$
        \State $\pushfront(\queue, r)$
      \EndIf
      \EndWhile
      \Return $\underline{I}$, $\U_I$
  \end{algorithmic}
\end{algorithm}
The algorithm is a loop that operates on a queue ($\queue$) of
intervals. An element of the queue is retrieved (and removed) from the
front of the queue using the $\pop$ operation, while an element is
added to the front of the queue using the $\pushfront$ operation. The
following steps are implemented until the queue is empty. Every
interval $[x]$ is checked to see whether it may be part of the pre-set
of $\Omega$. Two tests are performed:
\begin{enumerate}
\item Is the forward reachable set from $[x]$ disjoint with $\Omega$?
  If so, $[x]$ is disjoint from $I(\Omega)$.
\item Can an input $u$ be found so that the reachable set from $[x]$,
  restricted to the input $u$, lies entirely within $\Omega$? If so,
  $[x] \subset I(\Omega)$.
\end{enumerate}
If either condition is satisfied, then the set $[x]$ is saved in lists
labeled $\N$ (the collection of intervals disjoint with $I(\Omega)$)
and $\S$ (the collection of intervals contained in $I(\Omega)$).

If neither condition is satisfied and the $[x]$ is wider than the
specified tolerance, it is bisected along its largest dimension and
both resulting intervals are added to the front of the queue using the
$\pushfront$ operation. Otherwise, $[x]$ is added to $\mathcal{E}$,
which is a collection of the so-called ``indeterminate'' sets, which
are neither disjoint from nor subsets of $I(\Omega)$.

Line~\ref{line:u} of Algorithm~\ref{alg:I} requires checking the
condition
\begin{align}
  \label{eq:feas}
  \exists u \in \U \text{ s.t. } \overline{P}_u([x]) \subseteq \Omega,
\end{align}
which is a nonconvex feasibility problem, due to the nonconvexity of
$\Omega$. Luckily, we can exploit the structure of both
$\overline{P}_u([x])$ and $\Omega$ in order to efficiently and
precisely compute the set of $u$ that satisfy \eqref{eq:feas}.

Notice that, in the definition of $\overline{P}_u$, the term
$A[x] \oplus [\Phi]([x]) \oplus [\Psi]([x])\mathcal{U}$ is a convex
polyhedron, and the additive term $Su$ serves only to \emph{translate}
the resulting polyhedron. On the other hand, $\Omega$ is a union of
intervals, and also more generally a union of polyhedra. Therefore, we
can reduce the feasibility problem in \eqref{eq:feas} to a problem of
translating a polyhedron into a union of polyhedra. We describe here
an equivalence which helps solve this problem, inspired by
\cite{baker1986polygon}.

\begin{lemma}\label{lem:u}
  Let $\P = \conv(\Omega \cap \overline{P}([x]))$ and
  $\Q = \P \setminus (\Omega \cap \overline{P}([x]))$. Then, the
  following statements are equivalent:
  \begin{enumerate}
    \item $\exists u \in \U \text{ s.t. } \overline{P}_u([x]) \subseteq \Omega$;
    \item $S\U([x]) \triangleq \mathcal{I}(\overline{P}_0([x]), \P) \setminus \mathcal{O}(\overline{P}_0([x]), \Q)\ne \emptyset.$
  \end{enumerate}
\end{lemma}
\begin{proof}
  By defining
  $\overline{P}_0([x]) \triangleq A[x] \oplus [\Phi]([x]) \oplus
  [\Psi]([x])\mathcal{U}$, we obtain that for any $u$,
  $\overline{P}_u([x]) = \overline{P}_0([x]) \oplus
  \{Su\}$. Furthermore,
  $\overline{P}([x]) = \overline{P}_0([x]) \oplus S\U$.  This gives
  rise to the equivalance
  \begin{gather*}
    \exists u \in \U \text{ s.t. } \overline{P}_u([x]) \subseteq
    \Omega \\ \iff \exists r \in \real^n \text{ s.t. }
    \overline{P}_0([x]) \oplus \{r\} \subseteq \Omega \cap \overline{P}([x]),
  \end{gather*}
  where it must be true that $r = Su$. We see that by intersecting
  with $\overline{P}([x])$ on the right hand side, we can identify a
  translation by a vector $r$, which is automatically restricted to
  the range of $S$. To find the $r$ in the latter expression, we first
  find the translations into the \emph{convex hull} of
  $\Omega \cap \overline{P}([x])$ (i.e.,
  $\mathcal{I}(\overline{P}_0([x]), \P)$), then remove the
  translations that cause overlap with parts of the convex hull that
  are not in the original set (i.e.,
  $\mathcal{O}(\overline{P}_0([x]),\Q)$). This gives the set of
  translations $r$ that result in containment in
  $\Omega \cap \overline{P}([x])$, therefore yielding the expression
  for $S\U([x])$.
\end{proof}

Since $S\U([x])$ is the difference between a polytope and a union of
polytopes, it can be expressed as the union of a finite number of
polytopes. The equivalence in Lemma~\ref{lem:u} gives us a tractable
method of solving the feasibility problem in \eqref{eq:feas}, using
procedures to efficiently compute $\mathcal{I}$ and $\mathcal{O}$.
Finally, since $\forall r \in S\U([x]), \ \exists u \in \U$ such that
$r = Su$, we can recover the set of inputs with the Moore-Penrose
pseudoinverse, $\U([x]) = S^\dagger(S\U([x]))$, which is stored/saved
in $\U_I$ as pairs $([x],\U([x]) )$.

\begin{rem}
  \rm The computation of the set difference between two polyhedra can
  be represented as a union of polyhedra. This procedure is
  computationally expensive, especially when considering the
  difference between a polyhedron and a union of polyhedra, for which
  the best algorithms \cite{baotic2009polytope} have exponential
  complexity. In fact, the computation of set differences is the
  largest computaional burden of our method. Caching of results from
  previous iterations may reduce the number of differences that must
  be computed in the final implementation, although these details are
  not described here for the sake of brevity. \oprocend
\end{rem}

We finish this section by proving that Algorithm~1 returns a useful
approximation of $I(\Omega)$, which is within some known bound of the maximal CIS.
\begin{lemma}
  \label{lem:Convergence}
  Let $\Omega$ be a finite union of compact intervals. Then, for any
  precision $\varepsilon > 0$, Algorithm~\ref{alg:I} terminates in a
  finite number of iterations. Furthermore, letting
  $\underline{I}(\Omega)$ denote the output of Algorithm~\ref{alg:I},
  it holds that
  \begin{align}
    I(\Omega \ominus \mathcal{B}_{\rho\varepsilon}) \subseteq I_{\rho\varepsilon}(\Omega) \subseteq \underline{I}(\Omega) \subseteq I(\Omega).
  \end{align}
\end{lemma}

\begin{proof}
  The proof is similar to \cite[Lemma 1]{li2018invariant}, relying on
  the bounds provided by Lemma~\ref{lem:widths}.
\end{proof}

\subsection{Near-Maximal Controlled Invariant Sets}
With the ability to compute $\underline{I}(\Omega)$ using
Algorithm~\ref{alg:I}, all that remains is to repeat this operation
until convergence is achieved. Algorithm~\ref{alg:I-infty} describes
this simple procedure, which is guaranteed to terminate in finite time.
\begin{algorithm}
  \caption{Approximation of $I^\infty(\Omega)$}\label{alg:I-infty}
  \begin{algorithmic}[1]
    \Require $\Omega$, $\varepsilon$
    \State $I_0 \gets \Omega$, $I_1 \gets \emptyset$
    \While{$I_0 \neq I_1$}
    \State $I_1 \gets I_0$
    \State $I_0 \gets \underline{I}(I_0)$ \Comment via Algorithm~\ref{alg:I}
  \EndWhile
  \Return $I_0$
  \end{algorithmic}
\end{algorithm}

\begin{thm}\label{thm:RCIS} For any finite
  union of compact intervals $\Omega$ and precision $\varepsilon > 0$,
  Algorithm~\ref{alg:I-infty} terminates in a finite number of
  iterations. Furthermore, denoting the output of the algorithm as
  $\underline{I}^\infty(\Omega)$, the following inclusions hold:
  \begin{align*}
    I_r^\infty(\Omega) \subseteq \underline{I}^\infty(\Omega) \subseteq I^\infty(\Omega),
  \end{align*}
  where $r = \rho\varepsilon$. Finally, $\underline{I}^\infty(\Omega)$ is controlled invariant.
\end{thm}
Note that it is possible for the algorithm to return an empty set only
if the system does not admit an $r$-robustly controlled invariant set,
meaning $I_r^\infty = \emptyset$.
\begin{proof}
  Let $\underline{I}^k(\Omega)$ denote the value of
  $\underline{I}(\underline{I}^{k-1}(\Omega))$ in the $k$-th
  iteration, with $\underline{I}^0(\Omega) = \Omega$. Since the
  algorithm terminates if
  $\underline{I}^{k-1}(\Omega) = \underline{I}^k(\Omega)$, we only
  need to consider the case when they are not equal.  In this case,
  the structure of Algorithm~\ref{alg:I} is such that
  $\underline{I}^{k}(\Omega) \subsetneq \underline{I}^{k-1}(\Omega)$.
  Since $\Omega$ is compact, $\underline{I}^k(\Omega)$ contains only a
  finite number of intervals which will be considered using the
  bisection method for a given $\epsilon$.  Then $\exists K \ge 0$
  such that $\underline{I}^k(\Omega) = \emptyset, \ \forall k \ge K$,
  meaning the algorithm terminates in finite time.

  To prove the inclusions, note that by Lemma~\ref{lem:Convergence},
  $I_r(\Omega) \subseteq \underline{I}(\Omega) \subseteq
  I(\Omega)$. Also, for any two sets
  $A \subseteq B \subseteq \real^n$, we know that
  $\underline{I}(A) \subseteq \underline{I}(B)$ and
  $I_r(A) \subseteq I_r(B)$. Therefore by applying
  Lemma~\ref{lem:Convergence} and induction, we can determine that
  $I_r^k(\Omega) \subseteq \underline{I}^k(\Omega) \subseteq
  I^k(\Omega)$ for all $k \ge 0$.

  Finally, to prove invariance of $\underline{I}^\infty(\Omega)$, note
  that by Lemma~\ref{lem:Convergence},
  $\underline{I}^\infty(\Omega) =
  \underline{I}(\underline{I}^\infty(\Omega)) \subseteq
  I(\underline{I}^\infty(\Omega))$.
\end{proof}

\begin{rem}
  [An Outside-In Approach] \label{rem:Outside_In}\rm The method
  described in this paper is commonly known as an \emph{outside-in}
  method, since we start with a region of interest $\Omega$ and search
  for the largest invariant set $\mathcal{X}=I^\infty(\Omega)$ which it
  contains. It is also possible to use an \emph{inside-out} approach,
  which starts with a known invariant set, and iteratively grows that
  set \cite{bravo2005invariant}. Our method can be easily adapted for
  this case, as well. \oprocend
\end{rem}
To conclude this section, we state a result that has the potential to
significantly reduce the number of iterations required by
Algorithm~\ref{alg:I-infty}, hence, accelerating convergence.  Let
$\{\Omega_i\}_{i=1}^N$ be any \emph{ordered} collection of sets
such that $\Omega = \bigcup_{i=1}^N \Omega_i$.  Define the operator
\begin{align*}
  I_i(\Omega) \triangleq \bigcup_{j \ne i} \Omega_j \cup \left(Q(\Omega) \cap \Omega_i\right),
\end{align*}
and let $I_{1:N} \triangleq I^N(\cdots (I^1(\Omega))$. We let the
$I_i$ operator preserve the order of the constituent subsets by
defining
\begin{align*}
  (I_i(\Omega))_i = Q(\Omega)\cap\Omega_i
  \quad \text{and} \quad
  (I_i(\Omega))_j = \Omega_j, \ j \ne i.
\end{align*}
\begin{thm}\label{thm:partial}
  For a closed set $\Omega \subseteq \real^n$, $I_{1:N}$ satisfies
  \begin{gather}
  I_{1:N}(\Omega) \subseteq I(\Omega), \text{ and } \nonumber
  \\ \lim_{k\to\infty} (I_{1:N})^k(\Omega) = \lim_{k\to\infty}
  I^k(\Omega) = I^\infty(\Omega).\label{eq:convergence}
  \end{gather}
\end{thm}
\begin{proof}
  By definition, for any $i \ne j$,
 \begin{align*}
    I_i(I_j(\Omega)) &= \bigcup_{\ell \ne i} I_j(\Omega)_\ell \cup
    \left(Q(I_j(\Omega)) \cap I_j(\Omega)_i\right) \\ &= \hspace{-0.1cm}\bigcup_{\ell \ne
      i, \ell \ne j} \hspace{-0.15cm} \Omega_\ell\,  \cup \left(Q(\Omega) \cap \Omega_j\right) \cup \left(Q(I_j(\Omega)) \cap
      I_j(\Omega)_i\right)
    \\ &\subseteq \bigcup_{\ell \ne i, \ell \ne j} \Omega_\ell \cup
    \left(Q(\Omega) \cap (\Omega_i \cup \Omega_j)\right),
  \end{align*}
  since $I_j(\Omega)_i \subseteq \Omega_i$ and
  $Q(I_j(\Omega)) \cap I_j(\Omega)_i \subseteq Q(\Omega) \cap
  \Omega_i$.  By applying this statement recursively, we arrive at
  \begin{align}\label{eq:subset}
    I_{1:N}(\Omega)
    \subseteq \bigcup_{i = 1}^N \Omega_i \cap Q(\Omega)
    = I(\Omega).
  \end{align}
  In \eqref{eq:convergence}, since $I_{1:N}$ is monotonically
  decreasing and closed, the limit exists.  We will show by induction
  that for all $k \ge 0$,
  $(I_{1:N})^k(\Omega) \supseteq I^\infty(\Omega)$. Consider a point
  $x \in I^\infty(\Omega)$. We know
  $x \in Q(I^\infty(\Omega)) \subseteq Q(\Omega)$. Therefore, by the
  definition of $I_{1}$, $x \in I_{1}(\Omega)$, meaning
  $I_{1}(\Omega) \supseteq I^\infty(\Omega)$. By repeating this
  reasoning for $I_2, \dots, I_N$, we can see that
  $I_{1:N}(\Omega) \supseteq I^\infty(\Omega)$. For the step case,
  assume that $(I_{1:N})^{k-1}(\Omega) \supseteq
  I^\infty(\Omega)$. Then by the same argument as the base case, for
  any $x \in I^\infty(\Omega)$,
  $x \in Q(I^\infty(\Omega)) \subseteq Q((I_{1:N})^{k-1}(\Omega))$,
  implying $(I_{1:N})^k(\Omega) \supseteq I^\infty(\Omega)$. Using
  this fact in combination with \eqref{eq:subset}, monotonicity of $I$
  and $I_{1:N}$, and an induction argument proves
  \eqref{eq:convergence}.
\end{proof}
Theorem~\ref{thm:partial} is useful because it allows us to
independently consider individual regions of the invariant set, rather
than approximating the $I$ operator all at once. Essentially, our
procedure can focus on $I_i(\Omega)$ without needing to compute all of
$Q(\Omega)$, which is more expensive. Using the $I_i$ operator often
leads to a reduction in the total number of iterations required.
Algorithm~\ref{alg:partial} describes the modified procedure that
takes advantage of this fact. In the algorithm, $\Omega$ is updated
whenever an interval is determined to not be a part of the invariant
set. The loop terminates if every interval in the queue has been
checked since $\Omega$ was last changed, meaning every interval is
contained in the invariant set. The function $\checked$ tracks whether
a given $[x]$ has been tested against the current $\Omega$. It
defaults to $0$, and is assigned a value of $1$ after $[x]$ is
processed. Note that $\pushback$ places an item in the back of the
queue.
\begin{algorithm}[h]
  \caption{Accelerated Approximation of $I^\infty(\Omega)$}\label{alg:partial}
  \begin{algorithmic}[1]
    \Require $\Omega$, $\varepsilon$
    \State $\queue \gets  \{\Omega\}$, $N \gets \emptyset$, $\underline{I} \gets \emptyset$, $\mathcal{E} \gets \emptyset$, $\U_I \gets \emptyset$
    \While{$\exists [x] \in \queue$, $\checked([x], \Omega) = 0$}
      \State $[x] \gets \pop(\queue)$
      \State Compute $A$, $\Phi$, $s_i$, and $\Psi_i$ on $[x]$
      \If{$\overline{P}([x]) \cap \Omega = \emptyset$}
        \State $N \gets N \cup [x]$
        \State $\Omega \gets \Omega \setminus [x]$
      \ElsIf{$\exists u \in \U \text{ s.t. } \overline{P}_u([x]) \subseteq \Omega$} \label{line:u2}
        \State $\checked([x], \Omega) \gets 1$
        \State $\pushback(\queue, [x])$
        \State $\U_I \gets \U_I \cup ([x],S^\dagger(S\U([x])))$
      \ElsIf{$w([x]) \le \varepsilon$}
        \State $\mathcal{E} \gets \mathcal{E} \cup [x]$
        \State $\Omega \gets \Omega \setminus [x]$
      \Else
        \State $(l, r) \gets \bisect([x])$
        \State $\pushfront(\queue, l)$
        \State $\pushfront(\queue, r)$
      \EndIf
    \EndWhile
    \State $\underline{I}^\infty \gets \bigcup_{[x]\in \queue} [x]$
    \Return $\underline{I}^\infty$, $\U_I$
  \end{algorithmic}
\end{algorithm}

\begin{lemma}
  For any finite union of compact intervals $\Omega$ and precision
  $\varepsilon > 0$, Algorithm~\ref{alg:partial} returns the same
  result as Algorithm~\ref{alg:I-infty}. Furthermore, the number of
  iterations required by Algorithm~\ref{alg:partial} is less than or
  equal to the number required by Algorithm~\ref{alg:I-infty}.
\end{lemma}

\begin{proof}
  The first statement can be proved by the same reasoning as
  Theorem~\ref{thm:RCIS}. The second statement arises from the
  inclusion $I_{1:N}(\Omega) \subseteq I(\Omega)$
  (cf. Theorem~\ref{thm:partial}) and can also be seen by direct
  comparison of the algorithms.
\end{proof}

\section{Examples and Comparisons}
In this section, we demonstrate the effectiveness of our approach on a
numerical example, i.e., an inverted pendulum on a cart. We also compare our
approach to the method for switched systems in \cite{li2018invariant},
where the input space is sampled/gridded and considered as controlled
modes.

As in \cite{li2018invariant}, we consider an inverted pendulum on a
cart, discretized using forward Euler with a sampling time of
$0.01 \text{s}$. The dynamics are
\begin{align*}
  \dot{x}_1 &= x_2, \\
  \dot{x}_2 &= \frac{mgl}{J}\sin(x_1) - \frac{b}{J}x_2 + \frac{l}{J}\cos(x_1)u,
\end{align*}
with parameters $m = 0.2$ kg, $g = 9.8$ $\text{m}/\text{s}^2$ ,
$l = 0.3$ $\text{m}$, $J = 0.006$ $\text{kg}\cdot \text{m}^2$, and
$b = 0.1$ $\text{N}/\text{m}/\text{s}$.  This system and its
discretization are control affine. We consider a region of interest
$\Omega = [-0.05, -0.05]\times[-0.01, 0.01]$ and an input set
$\U = [-0.1, 0.1]$.

\begin{figure}[h]
  \centering
  \includegraphics[width=\columnwidth]{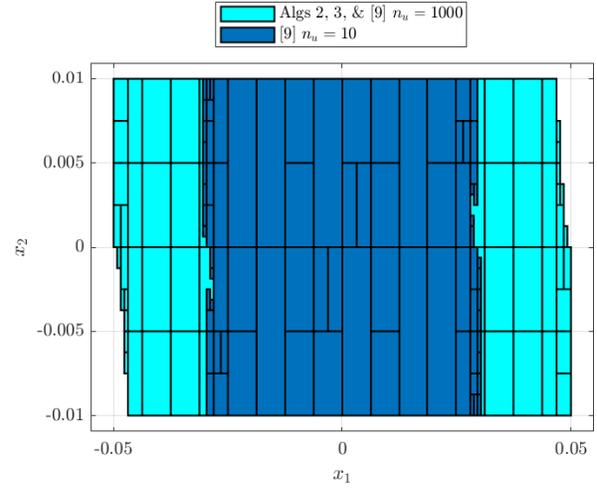}
  \caption{Controlled invariant sets of the inverted pendulum system.}
  \label{fig:x}
\end{figure}
\begin{figure}[h]
  \centering
  \includegraphics[width=\columnwidth]{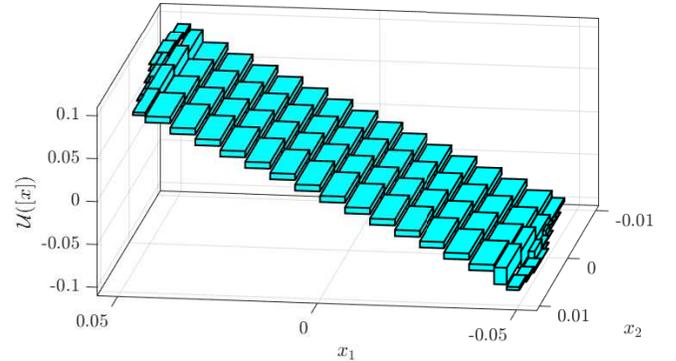}
  \caption{Union of invariance-enforcing inputs $\U([x])$ identified by
    Lemma~\ref{lem:u}.}
  \label{fig:u}
\end{figure}
Figure~\ref{fig:x} shows the identified controlled invariant sets for
our approaches, i.e., using Algorithms 2 and 3, and that of
\cite{li2018invariant}, with $n_u=10$ and $n_u=1000$ sampled
inputs. All methods were run with a precision $\varepsilon =
0.001$. On the other hand, Figure~\ref{fig:u} shows the union of all
invariance-enforcing control inputs $\U([x])$ identified by Algorithms
2 and 3. Finally, Table~\ref{tab:comp} shows a comparison of
computation times with different parameters.

\begin{table}[H]
  \centering
  \begin{tabular}[H]{c|c|c|c|c}
    Method & $\varepsilon$ & Iterations & Time (s) & Volume \\\hline \hline
    Algorithm~\ref{alg:I-infty} & $10^{-3}$ & 703 &  0.54 & 97.9\% \\
    Algorithm~\ref{alg:partial} & $10^{-3}$ & 618 & 0.63 & 97.9\% \\
    \cite{li2018invariant} ($n_u = 10$)  & $10^{-3}$ & 10729 & 0.12 & 59.8\% \\
    \cite{li2018invariant} ($n_u = 1000$)  & $10^{-3}$ & 485 & 0.59 & 97.9\% \\
  \end{tabular}
  \caption{Computational comparison between
    Algorithm~\ref{alg:I-infty}, Algorithm~\ref{alg:partial}, and
    \cite{li2018invariant}. $n_u$ denotes the number of input samples
    taken uniformly across $\U$. The volume is given in \% of the
    original $\Omega$.}
  \label{tab:comp}
\end{table}
Evidently, our method is able to identify a larger CIS in fewer
iterations than the sampling and interval arithmetic based approach in
\cite{li2018invariant}, when the number of samples is small. This is
presumably due to the higher accuracy of our polytopic approximations,
and the fact that we consider the entire continuous range of control
inputs. Increasing the number of sampled inputs results in a better
approximation of the CIS, at the cost of some additional computation
time.

\balance
\section{Conclusion}
We proposed two methods for approximating controlled invariant sets of
nonlinear control-affine systems using an iterative refinement
approach. We used techniques from computational geometry involving
translations of polyhedra to allow us to efficiently compute
continuous sets of feasible control inputs, rather than using a
sampling approach with switched dynamics. We demonstrated the
effectiveness of our method on a numerical example, which
showed improved accuracy over existing methods and led to faster
convergence in some cases. In the future, we will further explore the
extension of our approach to continuous time, as well as the control
synthesis problem, including some notions of optimality, while also
investigating ways to improve the accuracy and efficiency of our
algorithms. We also will test our approaches on a wide variety of
nonlinear systems.

\appendix
\subsection{Proof of Proposition~\ref{prop:touch}}\label{sec:prop-proof}
We begin by stating two intermediate results which will be used to
prove the proposition. The first allows us to determine whether two
polytopes intersect by examining the hyperplanes defining each
polytope.
\begin{prop}\label{prop:hyper}
  Given two polytopes $\P$ and $\Q$, $\P \cap \Q \ne \emptyset$ if and
  only if both of the following statements are true
  \begin{enumerate}
  \item $\P$ intersects every halfspace defining $\Q$, i.e.,
    $\forall i \in \until{N_\Q}, \ \exists p_i \in \P$ such that
    $(H_\Q)_i p_i \le (b_\Q)_i$.
  \item $\Q$ intersects every halfspace defining $\P$, i.e.,
    $\forall i \in \until{N_\P}, \ \exists q_i \in \Q$ such that
    $(H_\P)_i q_i \le (b_\P)_i$.
  \end{enumerate}
\end{prop}
\begin{proof}
  Necessity is simple, since if $\P \cap \Q \ne \emptyset$, every
  $v \in {\P \cap \Q}$ will satisfy the existence conditions in 1) and
  2).

  To prove sufficiency, note that $\P \cap \Q = \emptyset$ if and only
  if there exists a separating hyperplane, defined by some
  $h_* \in \real^n$ and $b_* \in \real$, such that
  $\forall v \in \P, \ h^\top v \le b$ and
  $\forall v \in \Q, \ h^\top v > b$. Conditions 1) and 2) preclude
  the existence of this separating hyperplane, implying
  $\P \cap \Q \ne \emptyset$.
\end{proof}

The second intermediate result tells us how to translate a polytope
so that it intersects a given halfspace.
\begin{prop}\label{prop:half-touch}
  Given a polytope $\P$ and halfspace $\H = \{x : h^\top x \le b\}$,
  the set of translations, i.e., of $\P$ that intersect $\H$; i.e.
  $\mathcal{O}(\P, \H) \triangleq \{s \in \real^n : \P \oplus \{s\}
  \cap \H \ne \emptyset \}$, is given by
\begin{align*}
  \mathcal{O}(\P, \H) = \{s \in \real^n : \ h^\top s \leq b - \alpha \},
\end{align*}
where $\alpha = \min_{v\in V_\P} h^\top v$.
\end{prop}
\begin{proof}
  The reasoning is similar to \cite[Theorem 2.3]{gilbert1998}, with
  $\min$ replacing $\max$ because intersection, rather than
  containment, is required.
\end{proof}
Proposition~\ref{prop:touch} follows from the combination of
Proposition~\ref{prop:hyper} with repeated application of
Proposition~\ref{prop:half-touch} to every halfspace defining both
$\P$ and $\Q$ (with $-s$ replacing $s$ in the second part, as the
translation is applied to $\P$, not $\Q$).


\bibliographystyle{ieeetran}
{
  \bibliography{biblio}
}

\end{document}

%% file: macros-abbrev.tex
\newcommand{\integernonnegative}{\ensuremath{\mathbb{Z}}_{\ge 0}}
\newcommand{\real}{\ensuremath{\mathbb{R}}}

\newcommand{\until}[1]{\{1,\dots, #1\}}

\renewcommand{\epsilon}{\varepsilon}

\usepackage{psfrag}